\documentclass[12pt]{amsart}

\usepackage[utf8x]{inputenc}
\usepackage{amssymb,latexsym}
\usepackage{amsmath}
\usepackage{graphicx}

\newtheorem{theorem}{Theorem}[section]
\newtheorem{lemma}[theorem]{Lemma}

\newtheorem{conjecture}{Conjecture}[section]

\newtheorem{claim}[theorem]{Claim}

\newtheorem{question}[theorem]{Question}

\renewcommand{\geq}{\geqslant}
\renewcommand{\leq}{\leqslant}
\renewcommand{\ge}{\geqslant}
\renewcommand{\le}{\leqslant}

\def\cref#1{Corollary~$\ref{#1}$}

\def\b1{\bar{1}}
\def\cb1{\cdot \bar{1}}

\usepackage{array,color,colortbl}

\title{Tournaments and the Strong Erd\H{o}s-Hajnal Property}

\author[E. Berger]{Eli Berger}
\author[K. Choromanski]{Krzysztof Choromanski}
\author[M. Chudnovsky]{Maria Chudnovsky}\thanks{Maria Chudnovsky was supported by NSF grant DMS-1763817. 
This material is based upon work supported in part by the U. S. Army
Research Laboratory and the U. S. Army Research Office under grant
number W911NF1610404.}
\author[S. Zerbib]{Shira Zerbib}\thanks{Shira Zerbib was supported by  NSF grant DMS-1953929.\\ \indent Eli Berger, Maria Chudnovsky and Shira Zerbib were supported by  US-Israel 
BSF grant 2016077.}

\date{}

\begin{document}

\maketitle
\begin{abstract}
A conjecture of Alon, Pach and Solymosi, which is equivalent to the celebrated Erd\H{o}s-Hajnal Conjecture, states  that for every tournament $S$ there exists $\varepsilon(S)>0$ such that if $T$ is an $n$-vertex tournament that does not contains $S$ as a subtournament, then $T$ contains a transitive subtournament on at least $n^{\varepsilon(S)}$ vertices.  
Let $C_5$ be the unique five-vertex tournament where every vertex has two inneighbors and two outneighbors.
The Alon-Pach-Solymosi conjecture is known to be true for the case when  $S=C_5$. 
Here we prove a strengthening of this result, showing that
in every tournament $T$ with no subtorunament isomorphic to $C_5$ there exist
disjoint vertex subsets $A$ and $B$, each containing a linear proportion of the vertices of $T$,
and such that every vertex of $A$ is adjacent to every vertex of $B$. 
\end{abstract}

\section{introduction}
A {\em tournament} is a complete graph with directions on edges. A tournament is {\em transitive} if it has no directed triangles. For tournaments $S,T$ we say that $T$ is {\em $S$-free} if no subtournament of $T$ is 
isomorphic to $S$.  
In \cite{APS} a conjecture was made
 concerning tournaments with a fixed forbidden subtournament:
 
 \begin{conjecture}
\label{EHtourn}
 For every tournament $S$ there exists $\varepsilon>0$ such that every $S$-free  $n$-vertex tournament
contains a transitive subtournament on at least $n^{\varepsilon}$ vertices. 
 \end{conjecture}
 
 It was shown in \cite{APS} that Conjecture~\ref{EHtourn} is equivalent to the Erd\H{o}s-Hajnal Conjecture \cite{EH1,EH2}.
Conjecture~\ref{EHtourn} is known to hold for a few types of tournaments $S$ \cite{EHtourn, EHsix}, but is still wide open in general.
 
 Let $D$ be a directed graph, and let $A, B \subseteq V(D)$ with $A \cap B =\emptyset$.
We say that $A$ is {\em complete to}  $B$ if all every vertex of $A$ is adjacent to every vertex of $B$,
and that $A$ is {\em complete  from} $B$ if every vertex of $A$ is adjacent from every vertex of $B$.
  A class of tournaments is {\em hereditary} if it is closed under subtournaments. 
 A hereditary class of tournaments $\mathcal{T}$ has the 
{\em strong Erd\H{o}s-Hajnal property} if there exists $\varepsilon = \varepsilon(\mathcal{T})$ such that for every $T\in \mathcal{T}$ there exist disjoint subsets $A$, $B$ of $V(T)$, each of size  $\varepsilon|V(T)|$ such that $A$ is complete to  $B$. 

The following question is closely related to Conjecture \ref{EHtourn}.
\begin{question}
\label{strongEH}
 For  which tournaments $S$ does the class of $S$-free tournaments have the strong
 Erd\"{o}s-Hajnal property?
 \end{question}

It is easy to see \cite{EHtourn} that if $S$ is a tournament, and the class of $S$-free graphs has the strong Erd\"{o}s-Hajnal property, then 
\ref{EHtourn} is true for $S$. In \cite{trees} there is a list of necessary conditions for a tournament $S$ to satisfy  \ref{strongEH}.

Denote by $C_5$ the (unique) tournament on 5 vertices in which every vertex is adjacent to exactly two other vertices. One way to construct this tournament is with vertex set $\{0,1,2,3,4\}$ and $i$ is adjacent to $i+1  \mod 5$ and $i+2 \mod 5$ (see Figure \ref{figC5}).

\begin{figure}
\begin{center}
\includegraphics[width=4in]{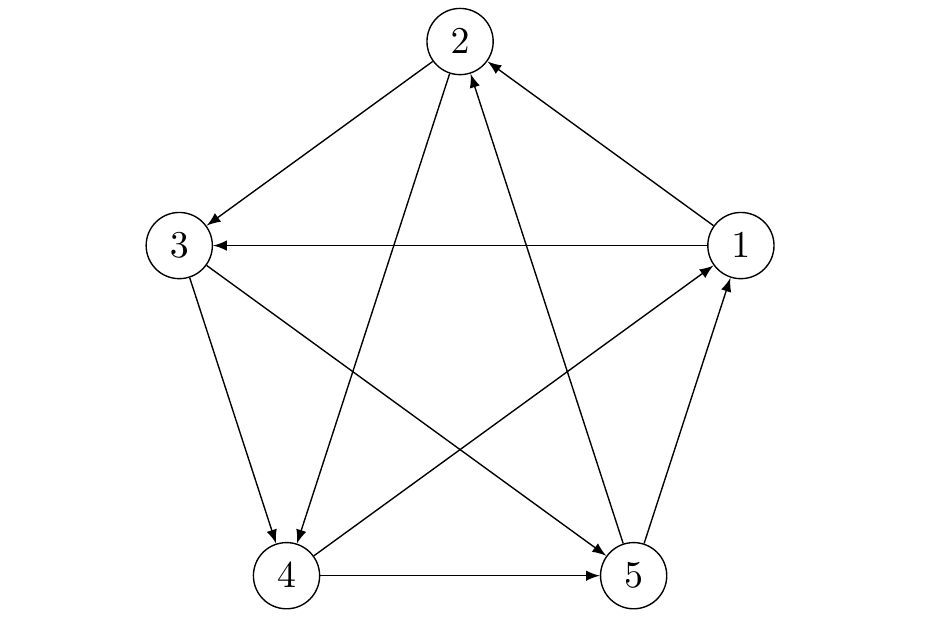}
\end{center}
\caption{The $C_5$ tournament.}
\label{figC5}
\end{figure}

In \cite{EHtourn} it was proved that $C_5$ satisfies the Erd\H{o}s-Hajnal conjecture. Here we prove the following stronger result:

\begin{theorem}\label{mainC5}
The class of $C_5$-free tournaments has the strong Erd\H{o}s-Hajnal  property.
\end{theorem}

\section{Regularity Tools} \label{regularity}

We recall some definitions  given in \cite{EHsix}.

Let $c>0$, $0<\lambda<1$ be constants, and let $w$ be a $\{0,1\}$-vector of length $|w|$. Let $T$ be a tournament with $|V(T)|=n$. 
Denote by $tr(T)$ the largest size of the transitive subtournament of $T$. 
For $\varepsilon>0$ we call a tournament $T$  \textit{$\varepsilon$-critical} for $\varepsilon>0$ if $tr(T) < |T|^{\varepsilon}$ but for every proper subtournament $S$ of $T$ we have: $tr(S) \geq |S|^{\varepsilon}$.
A sequence of disjoint subsets 
$(S_{1},S_{2},...,S_{|w|})$ of $V(T)$ is a $(c, \lambda, w)$-{\em structure} if
\begin{itemize} 
\item whenever $w_i=0$ we have $|S_{i}| \geq cn$,  
\item whenever $w_i=1$ the subtournament induced on $S_{i}$ is transitive and $|S_{i}| \geq c \cdot tr(T)$,  
\item $d^+(S_{i},S_{j}) \geq 1 - \lambda$ for all $1 \leq i < j \leq |w|$.
\end{itemize}

We say that a $(c,\lambda,w)$-structure is \textit{smooth} if the last condition of the definition of the
$(c,\lambda,w)$-structure is satisfied in a stronger form, namely we have:
for every $i<j$, every $v \in S_i$ has at most $\lambda|S_j|$ inneighbors in $S_j$, and
every $v \in S_j$ has at most $\lambda|S_i|$ outneighbors in $S_i$.

Theorem 3.5 of \cite{EHsix} asserts:

\begin{theorem}
\label{smooththeorem}
Let $S$ be a tournament, let $w$ be a $\{0,1\}$-vector, and let 
$0 < \lambda < \frac{1}{2}$ be a constant. Then there exist $\varepsilon,c>0$ such that every $S$-free $\varepsilon$-critical tournament contains a smooth $(c, \lambda, w)$-structure.
\end{theorem}

 Here we need a weaker form  of Theorem~\ref{smooththeorem} for the case when $w$ is the all-zero vector. 
 It turns out that in that case we do  need the criticality assumption. The proof consists of standard regularity lemma arguments, and can be easily reconstructed from the proof of \ref{smooththeorem} and 2.8  in \cite{EHtourn}. Thus we have:

\begin{theorem}\label{structurexists}
Let $S$ be a $k$-vertex tournament and  let $w$ be an  all-zero vector. There exists $c>0$ such that  every $S$-free tournament  contains a smooth $(c, \frac{1}{k}, w)$-structure.
\end{theorem}

\section{A lemma on outsimplicial directed graphs}

We say that a directed graph is {\em outsimplicial} if the outneighborhood of each vertex is a clique in the underlying undirected graph.

Next we prove the main lemma that we use for the proof of  Theorem \ref{mainC5}. The ideas in the proof of
Lemma~\ref{outsimp} seem to go beyond the particular setup of the lemma. In fact, they were later used in
\cite{PP10} to prove the analogue of Theorem~\ref{mainC5} to three $6$-vertex tournaments containing
$C_5$.

\begin{lemma}\label{outsimp}
If $D$ is an outsimplicial directed graph on $n>1$ vertices, then there exist two disjoint subsets  $A$ and $B$ of $V(D)$, both of size at least $\lfloor n/6 \rfloor$, such that either
\begin{enumerate}
\item[(i)] there is no edge, in any direction, between a vertex of $A$ and a vertex of $B$, 
or 
\item[(ii)] there is a path from every vertex of $A$ to every vertex of $B$. 
\end{enumerate}
\end{lemma}
\begin{proof}
Assume this is false. Then $D$ is not strongly connected, and, moreover, every strongly connected component $C$ of $D$
has size at most $n/3$, for otherwise a balanced partition of $C$ satisfies (ii).  
Let $C_1,\dots,C_m$ be the strongly connected components of $D$.
Let $F$ be the directed graph with vertex set $C_1,\dots,C_m$, and such that $C_i$ is adjacent to $C_j$ if and only if there is an edge from $C_i$ to $C_j$ in $D$. For $S \in V(F)$ let $w(S)= \sum_{C_i \in S}|V(C_i)|$. Note that $F$ is an acyclic directed graph. Let $F'$ be the underlying undirected graph of $F$.
\\
\\
{\em (1) $F$ is outsimpicial.}
\\
\\
To see this, suppose to the contrary that $C_i$ is adjacent to $C_j$ and to $C_k$ but there is no edge from $C_j$ to $C_k$ in $F'$.
Then in $D$, no vertex of $C_j$ is adjacent  to or from  a vertex of $C_k$. 
Let  $v_j,v_k \in C_i$, $u_j \in C_j$, $u_k \in C_k$ be such that $(v_j,u_j)$ and $(v_k,u_k)$ are edges in $D$. 
Since $D$ is outsimplicial and $u_j$ $u_k$ are not  in $D$, it follows that $v_j \neq v_k$. 
We may assume that 
$v_j$, $v_k$ are chosen so
that the directed path $P$ in $C_i$ from $v_j$ to $v_k$ is as short as possible.
Let $p$ be the outneighbor of $v_j$ in $P$. It follows from the minimality of
$P$ that $(p,u_k)$ is not an edge of $D$. Since $C_i$ and $C_k$ are distinct strongly connected components of $D$,
it follows that $(u_j,p)$ is not an edge. But now $p$ and $u_j$ are both outneighbors of $v_j$,
and there is no edge between then in either direction, contrary to the fact that $D$ is outsimplicial.
This proves (1).
\\
\\
{\em (2) $F'$ is chordal.}
\\
\\Indeed, suppose $C$ is an induced cycle of length larger than 3 in $F'$. Since $F$ is outsimplicial,
no vertex of $C$ has two outneighbors in $C$, and therefore $C$ is directed, a
contradiction to the fact that $F$ is acyclic. 
\\
\\
{\em  (3) No clique of $F'$ has weight $n/3$.} 
\\
\\
Indeed, suppose that $K$ be a clique of weight $n/3$. Then $K$ is a transitive subtournament of $F$.
Let $C_{k_1},\dots,C_{k_p}$ be the vertices of $K$ in the  transitive order. Then there is a path in $D$ from every vertex of $C_{k_i}$ to every vertex of $C_{k_j}$ for $i\le j$. For every $i\in[p]$, choose some order on the vertices in $C_{k_i}$, and consider the corresponding order $(C_{k_1}, \dots, C_{k_p})$ on $V_K=\bigcup_{i=1}^p C_{k_i}$. Let $A$ be the first $\frac{|V_K|}{2}$ vertices in this order and let $B=V_K\setminus A$. Then $A,B$ satisfy (ii), contradicting our assumption.  
This proves (3).
\\

Since $F'$ is chordal, $F'$ has a tree decomposition with bags being cliques of $F'$.
Let $(T,X)$ be such a tree decomposition, where the bag corresponding to a vertex $v$ of $T$ is denoted $X_v$. 

By Lemma 7.19 in \cite{CFK} there exists a bag $X_v$ of $T$ such that every connected component $D$ of $F' \setminus X_v$ has weight at most $w(V(F'))/2$. In particular, $F' \setminus X_v$ contains at least two connected components. 
Write $x = |V(F')\setminus X_v| \ge 2n/3$ (the inequality holds by (3)). 
Let $X_1,\dots, X_m$ be the connected components in $F' \setminus X_v$, ordered such that 
 $x_1 \le x_2 \le ... \le x_m$, where $x_i = |X_i|$. 
Then $x_i \le n/2$ for all $i$, and thus $m\ge 2$, showing $x_1\le x/2$. 
Write $y_i = \sum_{j=1}^i x_j$. 
Let $r$ be the index for which 
$y_{r-1} \le x/2$ and
$y_r > x/2$. 

Suppose first that $r=m$. Then $a=y_{r-1}$ and $b=x_r$ are both at least of size $n/6$, since 
$a = x-x_r\ge 2n/3 - n/2 =n/6$ and
$b = x-y_{r-1} \ge x - x/2 =x/2 \ge n/3.$ Therefore the sets $A = \bigcup_{i=1}^{r-1} X_i$ and $B=X_r$, of sizes $a$ and $b$ respectively, satisfy (i), a contradiction.  
It follows that  $r<m$; let $a=y_{r}$ and $b=x-y_r$. 
Note that $x_r \le x_{r+1}\le b$, and thus
$x/2 + 2b \ge y_{r-1}+ x_r + b =x$, showing
$b \ge x/4 \ge n/6$, and
$a \ge x/2 \ge n/3$ by the choice of $r$.  Therefore the sets $A = \bigcup_{i=1}^{r} X_i$ and $B=\bigcup_{i=r+1}^{m} X_i$, of sizes $a$ and $b$ respectively, satisfy (i), again a contradiction. Thus the lemma is proved. 
\end{proof}

\section{Proof of  Theorem \ref{mainC5}.} 
Let $T$ be a $C_5$-free tournament on $n$ vertices, and let $c$ be as in Theorem~\ref{structurexists} applied
with $S=C_5$. We show that there exist a constant $c$ and disjoint subsets $A$, $B$ of $V(T)$, such that $|A|=|B| = cn/6$ and    $A$ is complete to $B$. 

Assume to the contrary that this is false. 
Let $w$ be the zero vector of length $5$. 
By Theorem~\ref{structurexists} there exists $c>0$ and a smooth $(c, \frac{1}{5}, (0,0,0,0,0))$-structure  $\mathcal{S}=(V_{1},...,V_{5})$. By definition, we have that $|V_i| \ge cn$ for all $1\le i\le 5 $,  and for each $v_i \in V_i$, 
if $j>i$, then the number of inneighbors of $v_i$ in $V_j$ is at most $\frac{1}{5}|V_j|$, and
if $j<i$ then the number of outneighbors of $v_i$ in $V_j$ is at most $\frac{1}{5}|V_j|$. 
It follows that, if $j>i$, then the number of outneighbors of $v_i$ in $V_j$ is at least $\frac{4cn}{5}$, and
if $j<i$ then the number of inneighbors of $v_i$ in $V_j$ is at least $\frac{4cn}{5}$.

We now define a directed graph $D$ on the set of vertices $V_1$ in the following way: there is an edge between two vertices $u_1,v_1$ of $V_1$ if and only if they have a common inneighbor in $V_5$, and in this case the direction of the edge is the same as the direction of the edge between these two vertices in $T$.

\begin{claim}
If $(u_1,v_1)$ is an edge in $D$ then \begin{equation}\label{eq1}
    N^-(u_1) \cap V_3 \subseteq N^-(v_1) \cap V_3.
\end{equation}
\end{claim}
\begin{proof}
Assume $v_3 \in V_3$ is an inneighbor of $u_1$, but an outneighbor of $v_1$, and let $v_5 \in V_5$ be a common inneighbor of $u_1$ and $v_1$.  
Suppose first $(v_3, v_5) \in E(T)$. We claim that there exists a vertex $v_2 \in V_2$ such that $\{(v_2,v_3), (v_2,v_5), (u_1,v_2), (v_1,v_2)\} \subset E(T)$. Indeed,  the number of outneighbors of each of $u_1$ and $v_1$ in $V_2$ is at least $\frac{4cn}{5}$, and thus the number of common outneighbors of $u_1$ and $v_1$ in $V_2$ is at least $\frac{3cn}{5}$. Let $O_2 \in V_2$ be the set common outneighbor of $u_1,v_1$ in $V_2$. Since the number of outneighbors of $v_5$ in $O_2$ is at most $\frac{cn}{5}$, and the number of outneighbors of $v_3$ in $O_2$ is at most $\frac{cn}{5}$,  there must exist a vertex $v_2 \in O_2$ that is an inneighbor of both $v_3$ and $v_5$. Now $(u_1,v_1,v_2,v_3,v_5)$ form a $C_5$ subtournament of $T$, a contradiction.  

So assume $(v_5, v_3) \in E(T)$. Let $O_4 \in V_4$ be the set of common outneighbors of $u_1$ and $v_3$. Then, as before, $|O_4| \ge \frac{3cn}{5}$. Since the number of inneighbors
$v_5$ in $V_4$ is at least $\frac{4cn}{5}$, there exists a set $N_4 \subset O_4$ of size at least $\frac{2cn}{5}$ such that every vertex $v \in N_4$ has $(u_1,v),(v_3,v),(v,v_5) \in E(T)$. Similarly, there exists a set $N_2 \subset V_2$ of size at least $\frac{2cn}{5}$, such that every vertex $v\in N_2$ has $(u_1,v),(v,v_3),(v,v_5) \in E(T)$. 
If there exist $v_2\in N_2, v_4 \in N_4$ such that $(v_4,v_2) \in E(T)$ then $(v_3, u_1, v_4, v_2, v_5)$ is a $C_5$ subtournament of $T$, a contradiction. Otherwise $N_2$ is complete to $N_4$, contradicting our negation assumption.  
\end{proof}

\begin{claim}
 $D$ is an outsimplicial directed graph. 
 \end{claim}
\begin{proof}
In order to prove this, consider three vertices $u_1, v_1, w_1 \in V_1$ such that $(u_1, v_1), (u_1, w_1) \in E(D)$. We need to prove that there is a $D$ edge between $v_1$ and $w_1$ in some direction, i.e., that $v_1$ and $w_1$ have a common inneighbor in $V_5$. 
Let $x_5 \in V_5$ be the common inneighbor of $u_1$ and $v_1$ and 
let $y_5 \in V_5$ be the common inneighbor of $u_1$ and $w_1$.  
Assume for the
sake of contradiction that $x_{5}$ is an outneighbor of $w_{1}$.
Without loss of generality, $(x_5, y_5) \in E(T)$.
As before, there exists a set $O_3 \subset V_3$ of size at least $\frac{3cn}{5}$ such that every vertex $v\in O_3$ is an outneighbor of both $u_1$ and $w_1$, and there exists a set $I_3 \subset V_3$ of size at least $\frac{3cn}{5}$ such that every vertex $v\in I_3$ is an inneighbor of both $x_5$ and $y_5$. Thus there exists a vertex $z_3 \in O_3 \cap I_3$, and 
 $(u_1, w_1, z_3, x_5, y_5)$ is a $C_5$ subtournament in $T$, a contradiction. Therefore $x_5$ must also be an inneighbor of $w_1$, proving our claim.
\end{proof}

Now, by  Lemma \ref{outsimp} there exist sets $A$ and $B$ of $V(D)$, each  of size $|V(D)|/6 \geq cn/6$,  satisfying either (i) or (ii). 
Let $C$ be the set of vertices complete from $A$ in $T$. 

In case (i), there is no edge in $D$ between $A$ and $B$ in any direction, implying that no two vertices $a\in A$ and $b\in B$ have a common inneighbor in $V_5$. Thus the set $V_5 \setminus C$ is complete from $B$, and either $C$ or $V_5 \setminus C$ is of size at least $\frac{cn} {2}$, a contradiction.

In case (ii), there is a directed path from every vertex in $A$ to every vertex of $B$. 
Let $v \in V_3\setminus C$. Then $v$ is an inneighbor of some $a\in A$.  Let $b \in B$. There is a directed path $P$ in $D$  from $a$ to $b$. By  (\ref{eq1}), using induction on the length of $P$, $v$ is an inneighbor of every $p$ in $P$, and in particular $v$ is an inneighbor of $b$. 
It follows that $V_3\setminus C$ is complete to $B$.  Since  either $C$ or $V_3 \setminus C$ is of size at least $\frac{cn} {2}$, we get a contradiction. 
This concludes the proof of the theorem.

\end{document}